\documentclass[10pt,reqno]{amsart}
\usepackage{amsmath,amssymb,amsthm}
\usepackage{graphicx}
\usepackage{color}
\newtheorem{theorem}{Theorem}    % Standard theorem environment

\newtheorem{proposition}{Proposition}

\newtheorem{corollary}{Corollary}
\theoremstyle{definition}
\newtheorem{definition}[theorem]{Definition}

\newtheorem{remark}[theorem]{Remark}
\newtheorem*{remark*}{Remark}
\newcommand{\Z}{\mathbb{Z}}

\newcommand{\Q}{\mathbb{Q}}

\DeclareMathOperator{\breadth}{breadth}

\title[Knots with the same coefficient HOMFLY polynomials]{Knots whose braided satellite have the same HOMFLY polynomial up to given $z$-degrees}
\author[T.Ito]{Tetsuya Ito}
\address{Department of Mathematics, Kyoto University, Kyoto 606-8502, JAPAN}
\email{tetitoh@math.kyoto-u.ac.jp}

\begin{document}

\begin{abstract}
For a given knot $K$ and $w>0$, we construct infinitely many mutually distinct hyperbolic knots $\{K_i\}$ such that the $P$-satellites of $K$ and $K_i$ have the same HOMFLY polynomial up to given $z$-degrees, for all braided patterns $P$ with winding number less than or equal to $w$.
\end{abstract}

\maketitle

\section{Introduction}
Let $P_K(v,z)=\sum_{j=0}^{d}P^{2j}_K(v)z^{2j}$ be the HOMFLY polynomial of a knot $K$ defined by the skein relation
\[v^{-1}P_{K_+}(v,z) - vP_{K_-}(v,z)= zP_0(v,z), P_{\sf Unknot}(v,z)=1\]
We call $P^{2j}_K(v)$ the \emph{$2j$-th coefficient (HOMFLY) polynomial} of $K$.

Although it is widely open whether there exists a non-trivial knot with trivial HOMFLY polynomial, for a given knot $K$ and $N\geq 0$, there exist infinitely many knots whose $2j$-th coefficient polynomial is the same as $P^{2j}_K(v)$, for all $j=0,1,\ldots,N$ \cite{ka,mi}. In particular, there are infinitely many distinct knots having the trivial $2j$-th coefficient polynomial. Furthermore, the famous Kanenobu knots give infinitely many knots having the same HOMFLY polynomial \cite{kan}. 

Knots having the same polynomial invariant usually have different polynomial invariants after taking cables. For example, the $(2,1)$-cables of the Kanenobu knots have different $0$-th coefficient polynomial \cite{ta1}.

For the cabled invariant, an invariant for the cable of the original knot, similar problems remain widely open. Infinitely many knots whose cables have the same HOMFLY polynomial are not known. As far as the author knows, it is open whether there exist infinitely many distinct knots having the same $p$-colored Jones polynomial (the $p$-colored Jones polynomial is essentially the Jones polynomial of $(p-1,1)$-cable), even though it is not hard to find a pair of knots having the same colored Jones polynomial.

The aim of this paper is to give a construction of infinitely many knots whose braided satellites have the same coefficient polynomials up to given $z$-degrees.

A pattern $P=(V,k)$ is a pair consisting a solid torus $V=S^{1} \times D^{2}$ and a knot $k$ inside $V$. We say that $P$ is \emph{braided} if $k$ transverse all the meridional disks $\{pt\}\times D^{2}$ in the same sign, i.e., $k$ forms a closed braid inside $V$. The \emph{winding number} $w(P)$ is the algebraic crossing number of the meridional disk and $k$.

\begin{definition}
For a knot $K$, the \emph{$P$-satellite} $K_P$ is the knot $f(k)$, where $f: V \rightarrow S^{3}$ is a map that sends $V$ to a tubular neighborhood $N(K)$ of $K$ so that the meridian and the longitudes of $V$ are sends to the meridian and longitude of $K$, respectively.
\end{definition}

\begin{theorem}
\label{theorem:main}
For every knot $K$, $N>0$ and $w>0$,
there exists infinitely many mutually distinct hyperbolic knots $\{K_i\}$ such that  $P^{2j}_{K_P}(v)=P^{2j}_{(K_i)_P}(v)$ holds for all $j=0,1,\ldots, N$ and all braided patterns $P$ with $w(P)\leq w$.
\end{theorem}

In \cite{ta2} Takioka gave infinitely many non-trivial knots whose $(2,1)$-cable has the trivial $0$-th coefficient polynomial. The theorem gives the following far-reaching generalization.

\begin{corollary}
For every $p>0$, there exists infinitely many distinct hyperbolic knots whose $(q,1)$-cables have the trivial $2j$-th coefficient polynomial for all $j=0,1,\ldots, N$ and $q=1,\ldots,p$.
\end{corollary}

\section*{Acknowledgement}
The author is partially supported by JSPS KAKENHI Grant Numbers 19K03490, 21H04428, 	23K03110. 

\section{Finite type invariants determine the coefficient polynomials}

First we prove elementary algebraic observation that leads to a useful prospect of knot polynomial; \emph{sufficiently many finite type invariants determines the knot polynomials}. This is a bit surprising since a knot polynomial produces infinitely many distinct finite type invariants. This point of view, in a slightly different form, has been appeared in several places such as, \cite{dl,kss}. 

For a Laurent polynomial $F(v) \in \Z[v^{\pm 2}]$ and $i=0,1,\ldots,$ let $f_i \in \Q$ be the $i!$ times of the coefficient of $x^{i}$ of the formal power series obtained by putting $v=e^{x} = \sum_{n=0}^{\infty}\frac{x^{n}}{n!}$.
Namely, we define $f_i(K)$ by
\[ F(e^x)=\sum_{i=0}^{\infty} f_i\frac{x^i}{i!}\]
For $G(v) \in \Z[v^{\pm 2}]$, we define $g_i \in \Q$ similarly.

We define the \emph{breadth} of a Laurent polynomial $F(v) \in \Z[v^{\pm 2}]$ by 
\[ \breadth F = \begin{cases} 
\max \deg_v F - \min \deg_v F  & (F\neq 0) \\
-2 & (F=0)
\end{cases}
\]

\begin{proposition}
\label{prop:algebra}
For $F(v),G(v) \in \Z[v^{\pm 2}]$, if $f_i=g_i$ holds for all $0,\ldots,\frac{1}{2}\breadth F + \frac{1}{2}\breadth G + 1$, then $F(v)=G(v)$.
\end{proposition}

\begin{proof}
If $F(v)=G(v)=0$, we have nothing to prove, so we assume that $F(v)\neq 0$.

Let $m= \min \deg_v  F(v)$ and $m' = \min \deg_v G(v)$. Here if $F(v)=0$ we take $m'=m$. With no loss of generality, by interchanging the role of $F$ and $G$ if necessary, we can assume that $m\leq m'$.

Let $d=\frac{1}{2}\breadth F +  \frac{1}{2}\breadth G +2$. 
For an increasing sequence of integers $a_1<a_2<\cdots < a_d$, let $A(a_1,\ldots,a_d)$ be $d\times d$ Vandermonde matrix
\[ A(a_1,\ldots,a_d)= \begin{pmatrix}
1 & 1 & \cdots & 1 \\
a_1& a_2 & \cdots & a_d\\
a_1^2 & a_2^2 & \cdots& a_d^2 \\
\vdots & \vdots& \ddots & \vdots\\
a_1^{d-1} & a_2^{d-1}& \cdots & a_d^{n-1}
 \end{pmatrix}\]
We put $c_F(a_1,\ldots,a_d), f \in \Z^{d}$ by 
\[
c_F=c_F(a_1,\ldots,a_d)= \begin{pmatrix} c_{a_1}(F) \\ c_{a_2}(F) \\ \vdots \\ c_{a_d}(K) \end{pmatrix}, f = \begin{pmatrix} f_0(K) \\ f_1(K) \\ \vdots \\ f_{d-1}(K) \end{pmatrix}
\]
where $c_{a_i}(F)$ is the coefficient of $v^{a_i}$ of $F$. We define 
$c_{G}(a_1,\ldots,a_d), g \in \Z^{d}$ similarly.

Since $c_{a_i}(F) =0$ if $a_i<m$ or $m+\breadth F<i$, $c_F(a_1,\ldots,a_d)$ determines $F(v)$ if $\{a_1,\ldots,a_{d}\} \supset \{m,m+2,\ldots, m+\breadth(F)\}$. Similarly, $c_G(a_1,\ldots,a_d)$ determines $G(v)$ if $\{a_1,\ldots,a_{d}\} \supset \{m',m'+2,\ldots, m'+\breadth(G)\}$. 

Let $A=A(a_1,\ldots,a_d)$ where
\[ (a_1,\ldots,a_d) = (m,m+2,\ldots,m+2(d-1))\]
if $ m' \leq m +\breadth F$ and
\[ (a_1,\ldots,a_d) = (m,m+2,\ldots,m+\breadth F, m',m'+2,\ldots, m'+\breadth(G)) \]
If $m' > m + \breadth F$.
By the definition of $f_i$ and $g_i$, it follows that $Ac_F=f$ and $Ac_G=g$. Since $f=g$ by assumption, $A(c_F-c_G)=0$. The Vandermonde matrix $A$ is invertible so $c_F=c_G$. Therefore $F(v)=G(v)$.
\end{proof}

We apply the Proposition \ref{prop:algebra} for the coefficient polynomial $P^{2j}_K(v)$. We put
\[ P^{2j}_K(e^x)=\sum_{i=0}^{\infty} h^{2j}_i(K)\frac{x^i}{i!}\]
By the Morton-Franks-Williams inequality \cite{mo,fw}, we have
\[ \breadth P^{2j}_K(v) \leq \breadth_v P_K(v,z) \leq 2(b(K)-1) \] 
where $b(K)$ be the braid index of $K$.
Therefore  Proposition \ref{prop:algebra} proves the following.

\begin{theorem}
\label{theorem:HtoP}
If $h^{2j}_i(K)=h^{2j}_i(K')$ for all $i=0,\ldots,b(K)+b(K')-1$, then $P^{2j}_K(v)=P^{2j}_{K'}(v)$.
\end{theorem}

Two knots $K$ and $K'$ are said to be \emph{$n$-equivalent} if $v(K)=v(K')$ holds for all finite type invariants $v$ of degree $\leq n$. 
It is known that $h^{2j}_i$ is a finite type invariant of order $2j+i$, so it is convenient to rephrase Theorem \ref{theorem:main} in a following weaker form. 

\begin{corollary}
\label{cor:n-equivalence}
If $K$ and $K'$ are $n$-equivalent for $n\geq 2j +b(K)+b(K')-1$. Then $P^{2j}_K(v)=P^{2j}_{K'}(v)$.
\end{corollary}

As an immediate application, we get the following.

\begin{corollary}\cite[Theorem 3.1]{dl}
\label{cor:DL}
If $K$ is $n$-trivial (i.e. $n$-equivalent to the unknot) for $n\geq b(K)$, then the maximum self-linking number $\overline{sl}(K)$ is negative.
\end{corollary}
\begin{proof}
By Corollary \ref{cor:n-equivalence} $P^{0}_K(v)=P^{0}_{\sf unknot}(v)=1$. Hence by Morton-Franks-Williams inequality $\overline{sl}(K) +1 \leq \min \deg_{v} P_K(v,z) \leq \min \deg_v P^{0}_K(v) = 0$.
\end{proof}

Corollary \ref{cor:DL} makes a sharp contrast with \cite{ba} that says that for every knot $K$ and $n>0$, there exists a quasipositive knot which is $n$-equivalent to $K$. These results say that finite type invariants and the braid index give a restriction of quasipositivity, even though only one of them does not.

\section{Construction of knots sharing the same satellite coefficient polynomials}

By Corollary \ref{cor:n-equivalence}, a construction of knots having the same coefficient polynomial boils down to a construction of $n$-equivalent knots with  prescribed braid index bound.

\begin{proposition}
\label{prop:construction}
For a knot $K$, let $\beta \in B_m$ $(m\geq 3)$ be a braid whose closure is $K$
Let $\gamma \in \Gamma^n P_m$ and $K'$ be the closure of $\beta\gamma$, where $\Gamma^n P_m$ is the $n$-th lower central series of the pure braid group $P_m$. If $P$ is a braided pattern that satisfies $n \geq 2mw(P)+2N$, then the $P$-satellites $K_P$ and $K'_P$ have the same $2j$-th coefficient polynomial for $j=0,1,\ldots,N$.

\end{proposition}
\begin{proof}
By the construction, $K$ and $K'$ are $n$-equivalent \cite{ha,st}.
It is known that finite type invariants of order $\leq n$ of a (not necessarily braided) satellite knot $K_P$ is determined by the pattern $P$ and the finite type variants of order $\leq n$ of $K$ \cite{ku}. Thus $K_P$ and $K'_P$ are $n$-equivalent whenever $K$ and $K'$ are $n$-equivalent. Since $P$ is a braided pattern, $b(K_P),b(K'_P) \leq mw(P)$. Therefore $n \geq 2mw(P)+2N \geq b(K_P)+b(K'_P)+2N$ so by Corollary \ref{cor:n-equivalence} $P^{2j}_{K_P}(v)=P^{2j}_{K'_P}(v)$ holds for all $j=0,1,\ldots, N$. 
\end{proof}

It remains to show that the construction in Proposition \ref{prop:construction} can produce infinitely many mutually distinct hyperbolic knots.
We prove this by using the \emph{fractional Dehn twist coefficient} (\emph{FDTC}, in short).

The FDTC is a map $c:B_n \rightarrow \Q$ of the braid group $B_m$.
We refer to \cite{ik,hkm,ma} for definition and its basics.
We will only use the following properties of the FDTC.
\begin{itemize}
\item[(a)] $|c(\alpha \beta)-c(\alpha)-c(\beta)| \leq 1$ for all $\alpha,\beta \in B_m$ and $c(\alpha^{n})=nc(\alpha)$ for all $n \in \Z$.
\item[(b)] Let $K$ be a knot which is a closure of $\beta \in B_m$. Then $|c(\beta)|< g(K)+1$ holds \cite{it2,ik}. 
\item[(c)] Let $K$ be a knot which is a closure of $\beta \in B_m$. If $|c(\beta)|>1$ and $\beta$ is pseudo-Anosov, then $K$ is hyperbolic \cite{it1,ik}.
\item[(d)] Let $H$ be a non-trivial, non-central normal subgroup of $B_m$. Then for any $C>0$, there exists a pseudo-Anosov $\beta \in H$ such that $c(\beta)>C$ \cite[Theorem 3]{it3}.
\end{itemize}

\begin{proof}[Proof of Theorem \ref{theorem:main}]
Since $\Gamma^nP_{m}$ is a non-trivial, non-central normal subgroup of $B_m$, by the property (d) of FDTC we can take $\gamma \in \Gamma^{n}P_m$ so that $\gamma$ is pseudo-Anosov with $c(\gamma)> 1$. By the property (a) of FDTC, $\lim_{i \to \infty} c(\beta\gamma^i)=\infty$. Let $K_i$ be the closure of the braid $\beta\gamma^i$. By the property (b) of FDTC, if follows that $\lim_{i\to \infty}g(K_i) = \infty$ so the set $\{K_i\}$ contains infinitely many distinct knots.
Furthermore, since $\gamma$ is pseudo-Anosov, $\beta\gamma^{i}$ is pseudo-Anosov if $i$ is sufficiently large \cite{pa}. Therefore by the property (c) of FDTC, $K_i$ is hyperbolic whenever $i$ is sufficiently large.
\end{proof}

\begin{remark}
The construction in Proposition \ref{prop:construction} can be adapted so that the knots $K$ and $K'$ have the same Alexander polynomial and the Levine-Tristram signatures, as in \cite{ka,mi}.
We take $\gamma$ so that $\gamma \in \Gamma^{n} P_m \cap \mathrm{Ker}\, \mathcal{B}_m$ for $m \geq 5$, where $\mathcal{B}_m$ is the reduced Burau representation. Since the reduced Burau representation is non-faithful for $m\geq 5$ \cite{bi}, the intersection is a non-trivial, non-central normal subgroup \cite[Lemma 2.1]{lo}.
Both the Alexander polynomial and the Levine-Tristram signatures of closed braids are determined by the image of the Burau representation \cite{gg}, this shows that knots given in Theorem \ref{theorem:main} can be adjusted so that they share the same Alexander polynomial and the Levine-Tristram signatures.
\end{remark}

\section{Application for Kauffman polynomial}

We close the paper by giving applications of Proposition \ref{prop:algebra} to Kauffman polynomial. 

Let $D_K(v,z)=a^{-w(D)}\Lambda_{D}(v,z)$ be the Dubrovnik version of the Kauffman polynomial\footnote{Although we use the Dubrovnik polynomials, since $D_K(a,z) = F_K(-\sqrt{-1} v^{-1}, \sqrt{-1}z)$, all the results below hold for the Kauffman polynomial $F_K(v,z)$.}. Here $w(D)$ is the writhe and $\Lambda_D(v,z)$ is the regular
isotopy invariant defined by the skein relations
%\begin{align*}
%&\Lambda_{ \raisebox{-1mm}{\includegraphics*[width=4mm]{skeinc.eps}}}(v,z)
%= 1, 
%\qquad
% v^{-1} \Lambda_{ \raisebox{-1mm}{\includegraphics*[width=4mm]{skeincp.eps}}}(v,z)
%=\Lambda_{ \raisebox{-1mm}{\includegraphics*[width=4mm]{skeinst.eps}}}(v,z)
%= v\Lambda_{ \raisebox{-1mm}{\includegraphics*[width=4mm]{skeincn.eps}}}(v,z)\\
%&\Lambda_{ \raisebox{-1mm}{\includegraphics*[width=4mm]{skeinp.eps}}}(v,z)
%- \Lambda_{ \raisebox{-1mm}{\includegraphics*[width=4mm]{skeinn.eps}}}(v,z)
%= z\left(\Lambda_{ \raisebox{-1mm}{\includegraphics*[width=4mm]{skein8.eps}}}
%- \Lambda_{ \raisebox{-1mm}{\includegraphics*[width=4mm]{skein0.eps}}}\right),\\
%\end{align*}
\begin{align*}
&\Lambda_{\sf Unknot}(v,z) = 1, 
\qquad
v^{-1} \Lambda_{C_+}(v,z)=\Lambda_C(v,z)= v\Lambda_{C_-}(v,z)\\
&\Lambda_{K_+}(v,z) - \Lambda_{K_-}(v,z)
= z\left(\Lambda_{K_0}(v,z) - \Lambda_{K_{\infty}}(v,z)\right),
\end{align*}
where $C_{\pm}$ is the diagram obtained by adding positive or negative kink to the diagram $C$.
We put $D_K(v,z)=\sum_{j=0}^{d}D^j_{K}(v)z^j$ and we call $D^j_K(v) \in \Z[v^{\pm 1}]$ the \emph{$j$-th Kauffman coefficient polynomial} of $K$. It is known that $D^{j}_K(v) \in v^j\Z[v^{\pm 2}]$.
We put
\[ D^j_K(e^{x})=\sum_{i=0}^{\infty}k^{j}_i(K)\frac{x^{i}}{i!}\]
Then $k^{j}_i$ is a finite type invariant of order $i+j$.

For the Dubrovnik polynomial and the arc index $\alpha(K)$ we have the Morton-Beltrami inequality \cite{mb} 
\[ \textrm{breadth}_v D_K(v,z) \leq \alpha(K)-2. \]
Therefore by Proposition \ref{prop:algebra} we get the followings.

\begin{theorem}
\label{theorem:main-Kauffman}
If $k^{j}_i(K)=k^{j}_i(K')$ for all $i=0,1,\ldots,\frac{\alpha(K)}{2} +\frac{\alpha(K')}{2}-1$, then $D^{j}_K(v)=D^{j}_{K'}(v)$.
\end{theorem}

\begin{corollary}
If $K$ and $K'$ are $n$-equivalent for $n \geq j+ \frac{\alpha(K)}{2} +\frac{\alpha(K')}{2}-1$, then $D^j_K(v,z)=D^j_{K'}(v,z)$.
\end{corollary}

Let $c(K)$ be the minimum crossing number of $K$.
In \cite[Corollary 5.3]{dl} it is shown that if a knot $K$ is $c(K)$-trivial then $K$ has the trivial HOMFLY polynomial\footnote{We can also deduce this from Corollary \ref{cor:n-equivalence}.}.
The coefficient of $v^{i}z^j$ for the Dubrovnik polynomial $D_K(v,z)$ is non-trivial only if $|i|+j \leq c(K)$ \cite[Theorem 4]{th} hence
\[ \textrm{breadth}_v D^{j}_{K}(v) \leq 2(c(K)-j) \] 
holds. Therefore Proposition \ref{prop:algebra} gives the following.
\begin{corollary}
If $K$ is $c(K)$-trivial then $K$ has trivial Dubrovnik polynomial.
\end{corollary}

Unfortunately, these results are not effective to construct an example of infinitely many knots having the same (braided satellite) Kauffman polynomial up to given $z$-degrees. Unlike the braid index, it is impossible to construct infinitely many knots without increasing the arc index. This illustrates subtlety and difficulty of the problem to construct infinitely many knots having the same Kauffman polynomials.

\end{document}